\documentclass[12pt,reqno]{amsart}
\usepackage[left=2.4cm,top=2.8cm,right=2.4cm,bottom=2.8cm]{geometry}

\usepackage{amssymb,amsfonts,latexsym,amstext,epsfig}

\newtheorem{theorem}{Theorem}

\newtheorem{claim}[theorem]{Claim}

\newtheorem{corollary}[theorem]{Corollary}

\newtheorem{lemma}[theorem]{Lemma}

\newcommand{\Nnn}{{\mathbb N}}

\newcommand{\Zzz}{{\mathbb Z}}

\newcommand{\Rrr}{{\mathbb R}}
\newcommand{\va}{\mathbf{a}}
\newcommand{\vb}{\mathbf{b}}
\newcommand{\vu}{\mathbf{u}}
\newcommand{\vv}{\mathbf{v}}
\newcommand{\vw}{\mathbf{w}}
\newcommand{\vz}{\mathbf{z}}
\newcommand{\vx}{\mathbf{x}}
\newcommand{\vy}{\mathbf{y}}
\newcommand{\vp}{\mathbf{p}}
\newcommand{\vq}{\mathbf{q}}
\newcommand{\uu}{\mathbf{u}}

\begin{document}

\title[Infinite geometric graphs]{Infinite geometric graphs and properties of metrics}
\author{Anthony Bonato}
\address{Department of Mathematics\\
Ryerson University\\
Toronto, ON\\
Canada, M5B 2K3} \email{abonato@ryerson.ca}
\author{Jeannette Janssen}
\address{Department of Mathematics and Statistics\\
Dalhousie University\\
Halifax, NS\\
Canada, B3H 3J5}
\email{jeannette.janssen@dal.ca}
\thanks{The authors gratefully acknowledge support from NSERC and Ryerson University}

\begin{abstract}
We consider isomorphism properties of infinite random geometric graphs defined over a variety of metrics. In previous work, it was shown that for $\mathbb{R}^n$ with the $L_{\infty}$-metric, the infinite random geometric graph is, with probability 1, unique up to isomorphism. However, in the case $n=2$ this is false with either of the $L_2$-metric. We generalize this result to a large family of metrics induced by norms. Within this class of metric spaces, we show that the infinite geometric graph is not unique up to isomorphism if it has the new property which we name truncating: each step-isometry from a dense set to itself is an isometry. As a corollary, we derive that the infinite random geometric graph defined in $L_p$ space is not unique up to isomorphism with probability 1 for all finite $p>1.$
\end{abstract}
\maketitle

\section{Introduction}
Geometric random graph models play an emerging role in the modelling of
real-world networks such as on-line social networks \cite{bon1}, wireless
networks \cite{frieze}, and the web graph \cite{bon,FFV}. In such
stochastic models, vertices of the network are represented by points in a
suitably chosen metric space, and edges are chosen by a mixture of relative
proximity of the vertices and probabilistic rules. In real-world networks,
the underlying metric space is a representation of the hidden reality that
leads to the formation of edges. Such networks can be viewed as embedded in
a \emph{feature space}, where vertices with similar features are more closely positioned. For example,
in the case of on-line social networks, for
example, users are embedded in a high dimensional \emph{social space}, where
users that are positioned close together in the space exhibit similar
characteristics. The web graph may be viewed in \emph{topic space}, where
web pages with similar topics are closer to each other.
We note that the theory of random geometric graphs has been
extensively developed (see, for example, \cite{bol,ellis,walters}, \cite{goel}, and the books \cite{mr,mpen}).

The study of countably infinite graphs is motivated in part by the theory of the \emph{infinite random graph}, or the
\emph{Rado} graph (see \cite{cam,cam1,er}), written $R.$
The graph $R$ was first discovered by Erd\H{o}s and
R\'{e}nyi \cite{er}, who proved that with probability $1,$ any two randomly
generated countably infinite graphs, where vertices are joined independently
with probability $p\in (0,1)$, are isomorphic. The graph $R$ has several
remarkable properties, such as universality (all countable graphs are
isomorphic to an induced subgraph) and homogeneity (every isomorphism
between finite induced subgraphs extends to an isomorphism). The
investigation of $R$ lies at the intersection of logic, probability theory,
and topology; see \cite{cam,cam1,er} and Chapter 6 of \cite%
{bonato}.

In \cite{bon2} we considered infinite random geometric graphs. In our model,
the vertex set is a countable dense subset of $\mathbb{R}^{n}$ for a fixed $n\geq 1$ and two
vertices are adjacent if the distance between the two vertices is no larger
than some fixed real number. More precisely, consider a metric space $S$ with metric
\begin{equation*}
d:S\times S\rightarrow \mathbb{R},
\end{equation*}%
a parameter $\delta \in {\mathbb{R}}^{+},$ a countably infinite subset $V$ of $S$, and $p\in (0,1)
$. The \textsl{Local Area Random Graph $\mathrm{LARG}(V,\delta ,p)$} has
vertices $V,$ and for each pair of vertices $u$ and $v$ with $d(u,v)<\delta $,
an edge is added independently with probability $p$. The model can be similarly defined for $V$ finite.
For simplicity, we consider only the case when $\delta= 1$; we write
$\mathrm{LARG}(V,p)$ in this case. The LARG model generalizes well-known classes of
random graphs. For example, special cases of the $\mathrm{LARG}$ model
include the random geometric graphs, where $p=1$, and the binomial random
graph $G(n,p)$, where $S$ has finite diameter $D$, and $\delta \geq D$.


Let $\Omega_n$ be the set of metrics defined on $\mathbb{R}^n.$
We will say that a metric $d\in \Omega_n$ is \emph{unpredictable} if
 for all dense subsets  $V$ in $ \Rrr^n$ such that for all $p\in (0,1)$, with positive probability graphs generated by $\mathrm{LARG}(V,p)$ are non-isomorphic. The results of this paper show that a large class of metrics  $d\in \Omega_n$ are unpredictable. It was shown in \cite{bon2} that the $L_{2}$-metric is unpredictable, while the $L_{\infty}$-metric is not. For all $n\in
\mathbb{N}^+$ it was shown in \cite{bon2} that the $L_{\infty }$-metric is predictable, while the $L_{2}$-metric is not (that is, \emph{predictable}: there exists a countable dense set $V$ in $\Rrr^n$ such that for all $p\in (0,1)$, with probability $1$, any two graphs generated by $\mathrm{LARG}(V,p)$ are isomorphic).

We have therefore, the following natural classification program for infinite random geometric graphs.

\medskip\noindent
\textbf{Geometric Isomorphism Dichotomy (GID)}: Determine which metrics in $\Omega_n$ are unpredictable.
\medskip

In the next section, we describe our main results which greatly extend our understanding of the GID for a large family of so-called norm derived metrics, which includes all the familiar $L_p$-metrics and the hexagonal metric.
The hexagonal metric arises in the study of Voronoi diagrams and period graphs (see \cite{fu}) which have found applications to nanotechnology. We note that in \cite{bon3}, it was shown that the hexagonal metrics in the case $n=2$ are unpredictable.

All graphs considered are simple, undirected, and countable unless otherwise
stated.
Given a metric space $S$ with distance function $d$, denote the (open) \emph{ball of radius $\delta
$ around $x$} by
\begin{equation*}
B_{\delta }(x)=\{u\in S:d(u,x)<\delta \}.
\end{equation*}%
We will sometimes just refer to $B_{\delta }(x)$ as a \emph{$\delta$-ball} or \emph{ball of radius} $\delta$. A subset $V$ is \emph{dense} in $S$ if for every point $x\in S$, every ball around $x$
contains at least one point from $V$. We refer to $u\in S$ as points or
vertices, depending on the context. Throughout, let $\mathbb{N}$, $\mathbb{N}%
^{+}$, $\mathbb{Z},$ and $\mathbb{R}$ denote the non-negative integers, the positive integers, the integers, and real numbers,
respectively. We use \textbf{bold} notation for vectors $\vu \in \mathbb{R}^n.$ The dot product of two vectors $\vu$ and $\vv$ is denoted $\vu \cdot \vv$. For a reference on graph theory the reader is directed to \cite%
{diestel,west}, while \cite{bryant} is a reference on metric spaces.

\subsection{Main results}

Our main result is Theorem~\ref{non2} stated below, which settles the GID for a large class of metrics in the plane. Before we state the theorem, we first need some definitions. Throughout, $d$ is a metric in $\Omega_2.$ The definitions stated in this section naturally extend to higher dimensions and more general metric spaces, but we focus on the plane for simplicity, since our results apply there.

\subsubsection{The truncation property of metric spaces}\label{sec:trunc}

A function $f:V\rightarrow V$, where $V\subseteq \Rrr^2$ is a {\sl step-isometry} if for every $u,v\in V$, $$\lfloor d(u,v)\rfloor =\lfloor d(f(u),f(v))\rfloor.$$ Thus, a step-isometry preserves distances in truncated form.
A metric $d\in\Omega$ has the {\sl truncating property} if for all countable dense sets $V$ in $\Rrr^2$, every step-isometry $f:V\rightarrow V$ is an isometry. In \cite{bon2}, it was shown that $\Rrr^2$ with the $L_2$-metric has the truncating property, but the $L_\infty$-metric does not.
For example, in the case $n=1$ with the $L_{\infty}$-metric, consider the map $f:\Rrr\rightarrow \Rrr$ given by:
\[
f(x) =\left\{ \begin{array}{ll} \lfloor x\rfloor + \frac23 (x-\lfloor x\rfloor)&\mbox{if }x-\lfloor x\rfloor \leq \frac12
\vspace*{2mm},\\
\lfloor x\rfloor  +\frac43 (x-\lfloor x\rfloor) -\frac13 &\mbox{else. }\end{array}\right.
\vspace*{2mm}\\
\]
It is straightforward to check that $f$ is a step-isometry, but not an isometry.

Our motivations for considering step-isometries is that any isomorphism between graphs generated by the LARG model with $p\in (0,1)$ must be a step-isometry. We record this fact in the following lemma, which follows directly from Theorem~2.1 and Corollary~2.5 of \cite{bon2}.
\begin{lemma}\label{newlemmm}
For all dense sets $V$ and $p\in (0,1)$, if $f$ is an isomorphism between two graphs generated by the $\mathrm{LARG}(V,p)$, then with probability $1$, $f$ is a step-isometry.
\end{lemma}
For any dense set $V$ in $\Rrr^2$, and metric $d$, there are only a limited number of possibilities to form a bijective isometry from $V$ to $V$. This imposes strong restrictions on any possible isomorphism between graphs with vertex set $V$. This suggests a close relationship between the truncating property and unpredictability.  We will show that, for the cases we consider, the metrics which are unpredictable are indeed precisely those that have the truncating property.

\subsubsection{Norm-derived metrics and their shape}

A metric $d$ is \emph{translation-invariant} if for all $\mathbf{a}, \mathbf{x}, \mathbf{y} \in \Rrr^2$, we have that $$d(\vx ,\vy)= d(\vx + \va, \vy +\va).$$ The metric $d$ is \emph{homogeneous} if for all $\alpha \in \Rrr$, $d(\alpha\vx, \alpha\vy )=|\alpha | d(\vx,\vy)$. Any
translation-invariant, homogeneous metric $d$ induces a norm, given by  $\|\vx \| = d(\vx,\mathbf{0})$. Conversely, a norm $\|\cdot \|$ induces a translation-invariant, homogeneous metric $$d(\vx,\vy)=\|\vx - \vy \|.$$ The metrics we consider are precisely those that are derived from norms. We define a \emph{norm-derived} metric as a metric defined by a norm (referred to as the \emph{underlying norm}). A well-studied family of norm-derived metrics consists of metrics derived from the $L_p$ norm, where $p\geq 1$; recall that for $\vx = (x_1, x_2) \in \mathbb{R}^2$, $\| \vx \|_p = \left(|x_1|^p + |x_2|^p \right)^{1/p}.$

It is straightforward to check that the unit ball around $\mathbf{0}$ of any norm-derived metric $d$ must be a convex, point-symmetric set $P$ which we call the \emph{shape} of $d$.
Conversely, a convex, point-symmetric set $P\subseteq \Rrr^2$ with non-empty interior defines a norm
$\| \cdot \|_P$ and a corresponding metric $d_P$ as follows. Fix a vector $\vx\in\Rrr^2$, and let $\mathbf{b}$ be the unique point where the ray from $\mathbf{0}$ to $\vx$ intersects $P$. Then we have that
\[
\|\vx \|_P =\frac{\|\vx\|_2}{\|\mathbf{b}\|_2},
\]
and the metric $d_P$ is defined as $$d_P(\vx,\vy)=\|\vx-\vy\|_P.$$ Observe that the unit ball around $\mathbf{0}$ of $d_P$ equals $P$. Note that in $\Rrr ^2$, for $p>1,$ the $L_p$ metric has shape a superellipse or L\'ame curve, while the $L_\infty$-metric has shape a square with sides parallel to the coordinate axes, and the $L_1$-metric has shape a square with the diagonals parallel to the coordinate axes. Throughout this paper, we will use the notation $d_P$ for the norm-derived metric with shape $P$.

We only consider norm-derived metrics whose shape is either a polygon (and we call such metrics \emph{polygonal}; these include the $L_p$-metrics in the case $p=1,\infty$), or metrics whose shape can be described by a smooth curve (we call such metrics \emph{smooth}; they include the $L_p$ metrics in the case $p>1$ and $p\neq \infty$). We think that the results in this paper apply equally to norm-based metric whose shape has both straight and curved sides, and that this can be proved with methods similar to those presented here. However, we do not pursue this generalization here, due to the excessive technicalities. To conserve notation, \emph{from now on we denote by $\Omega$ the set of all both smooth and polygonal norm-derived metrics in $\Rrr^2$}.

\subsubsection{Main results}\label{subsec:main}

 We now have defined almost all concepts needed to state our main results, which constitute a classification of all metrics in $\Omega$ in terms of predicability and the truncating property. We next define a special class of polygonal metrics; our main result will show that this is the class of metrics in $\Omega $ which do not have the truncating property.
 A metric $d\in \Omega$ is a {\sl box metric} if its shape is a {\sl parallelogram}.

\begin{theorem}\label{thm:main}
A metric $d\in \Omega$ has the truncating property if and only if $d$ is not a box metric.
\end{theorem}
The proof of the theorem will follow from a series of lemmas presented in Section~2. The proof considers separately the polygonal and smooth cases.

An important corollary of Theorem~\ref{thm:main} shows that, in $(\Rrr^2,d_P)$ where $d_P$ is not a box metric, three well-chosen points completely determine a step-isometry, as they do an isometry. For example, if $d_P$ is the $L_2$-metric, then the restriction on the 3 points is that they should not be collinear. For general norm-derived metrics, the restriction is slightly more complicated, and we need some definitions before we can state the corollary.

A set of three points $\vx,\vy ,\vz\in\Rrr^2$ so that $d(\vx,\vy)\leq d(\vy,\vz)\leq d(\vx,\vz)$ is a {\sl triangular set} if $$d(\vx,\vy)+d(\vy,\vz) >d(\vx,\vz).$$
If $d$ is the $L_2$-metric, then $\vx,\vy,\vz$ forms a triangular set if and only if the points are not collinear; the same is not true for metrics where the unit ball $P$ is a polygon. For example, let $d_\infty$ be the metric derived from the $L_\infty $-norm, and let $\vx =(0,0)$, $\vy=(2,0)$, and $\vz=(1,1)$. Then $d_\infty(\vx,\vz)=d_\infty(\vy,\vz)=1$ and $d_\infty(\vx,\vy)=2$, so $\vx,\vy,\vz$ do not form a triangular set, despite the fact they are not collinear.

Triangular sets exist in any countable dense set and for any metric $d\in \Omega$ which is not a box metric. Namely, fix $\vx,\vy,\vz\in \Rrr^2$ such that the lines through $\vx$ and $\vy$, $\vy$ and $\vz$, and $\vz$ and $\vx$ have slope $\gamma_1$, $\gamma_2$, $\gamma_3$, respectively. If $d$ is a smooth metric, then if $\gamma_1$, $\gamma_2$, $\gamma_3$ are all different, then $\{ \vx,\vy,\vz\}$ is a triangular set. If $d$ is a polygonal metric, and if the lines through the origin with slopes $\gamma_1$, $\gamma_2$, $\gamma_3$ intersect $P$ in three different, non-parallel sides, then $\{ \vx,\vy,\vz\}$ is a triangular set.

It follows immediately from the properties of norm-derived metrics that triangular sets have the following \emph{anchoring property}: if $V$ is a set dense in $\Rrr^2$ and $S\subseteq V$ is a triangular set, and if $f:V\rightarrow V$ an isometry, then the images under $f$ of the points in $S$ completely determine the map $f$. (The argument can be found explicitly in the proof of Theorem~\ref{non2}.) Hence, images of triangular sets completely determine the images of points of $V$ under an isometry. Thus, it follows as corollary from Theorem~\ref{thm:main} that if $d$ is not a box metric, then triangular sets have the anchoring property if only the truncated distances are given.

\begin{corollary}\label{cor:3pts}
Let $V$ be a countable dense set and let $S\subseteq V$ be a triangular set, and let $d\in \Omega$ be a norm-based metric which is not a box metric.
Then for every step-isometry $f:V\rightarrow V$,  the images under $f$ of the points in $S$ completely determine $f$.

\end{corollary}

\begin{proof} By Theorem~\ref{thm:main}, the step-isometry $f$ is an isometry. The points of $S$ completely determine, therefore, the position of all points in $V$. The proof now follows since the image of a triangular set under an isometry is also a triangular set. \end{proof}
Thus, if the coordinates of the points in $S$ are given, and for all other points in $V$ the truncated distance to each of the points in $S$ is given, then the coordinates of all points in $V$ are determined.

Arguments involving infinite graphs can be used with Corollary \ref{cor:3pts} to prove the following
theorem, which settles the GID for the $L_p$-metrics and polygonal metrics on $\Rrr^2$.

\begin{theorem}\label{non2}
Let  $d$ be a metric in $\Omega$. Then $d$ is predictable if and only if it is a box metric.
\end{theorem}
We defer the proof to the end of Section~2. We conjecture that analogous results as in Theorem~\ref{non2} apply to higher dimensions and other norm-derived metrics. We will consider these cases in future work.

The theorem classifies all metrics in $\Omega$ in terms of predictability. We remind the reader that the definition of predictability  involves an existential quantifier: the definitions involve the behaviour of $LARG(V,p)$ on a dense subset $V$, where  $V$ must have special properties.  In fact, however, the conditions on $V$  needed to prove the theorem are rather mild. In fact, if $V$ is a random dense set, chosen according to a ``reasonable'' model, then $V$ will possess the desired properties with probability 1.  For example $V$ can be the countable union of sets each chosen from $\Rrr^2$ according to a Poisson point process. Another example is the extension to $\Rrr^2$ of the random dense set model proposed in \cite{t} for $(0,1)$, by first taking the union of an infinite number of such sets for the intervals $(z,z+1)$ where $z\in \Zzz$, and then taking the Cartesian product of two of such unions.



\section{Outline of the proof of the main results}\label{lemmass}

In this section, we will sketch the proofs of Theorems \ref{thm:main} and Theorem~\ref{non2}, and give the sequence of lemmas needed to arrive at the proofs.

\subsection{Norm-derived metrics}\label{sec:norm}

Before we proceed to an outline of the proof of Theorem~\ref{thm:main}, we first introduce some notation and concepts that apply to all norm-derived metrics, and will be helpful in the rest of the paper.  Note first that it follows directly from the definition that for any point-symmetric convex set $P$ with boundary $\mathcal{B}(P)$, and for any $\vx\in\Rrr^2$, we have that
\begin{equation*}
\frac{ ||\vx||_2}{\sup_{\vb\in \mathcal{B}(P)}||\vb||_2}\leq ||\vx ||_P\leq \frac{ ||\vx||_2}{\inf_{\vb\in \mathcal{B}(P)}||\vb||_2}.
\end{equation*}

For a polygonal metric $d_P\in \Omega$, we can use the description of the polygon to compute the norm and distance. We introduce specific notation for this case which we will use throughout.
Let $P$ be a point-symmetric (but possible non-regular) polygon in $\Rrr^2$ whose boundaries are formed by lines with normals $ \mathbf{a}_{1},\mathbf{a}_{2},\ldots ,\mathbf{a}_{k} $.   See Figure \ref{fig1} for an example.  Moreover, assume that the vectors $\va_i$ are scaled so that for each point $\vp$ on the boundary of $P$, $|\va_i\cdot \vp |= 1$ for some $1\leq i\leq k$. In this case, we can describe  $P$ to be the set:
\[
P=\{\mathbf{x}:\mbox{for all }1\leq i\leq k,\mbox{ }-1\leq \mathbf{a}
_{i}\cdot \mathbf{x}\leq 1\}.
\]

\begin{figure} [h!]
\begin{center}
\epsfig{figure=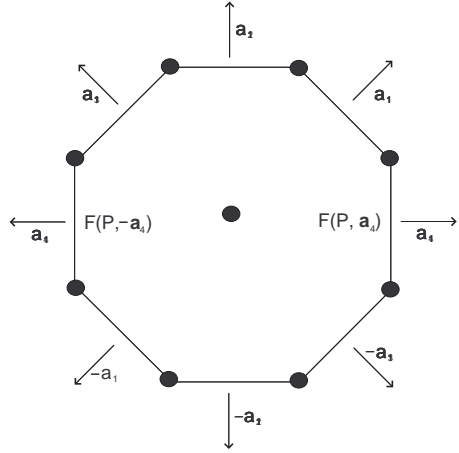,width=2.5in,height=2.5in}
\caption{The polygon $P$.}\label{fig1}
\end{center}
\end{figure}

Define the set of {\em generators} of $P$, written $\mathcal{G}_P$, to be the set of vectors that define its sides; in particular, we have that
\[
\mathcal{G}_P=\{ \mathbf{a}_i:1\leq i\leq k\}\cup \{-\mathbf{a}_i:1\leq i\leq k\}.
\]
Observe that an alternative description is $P=\{ \mathbf{x}: \mbox{for all }\mathbf{a}\in \mathcal{G}_P,\mbox{ }0\leq \mathbf{a}\cdot \mathbf{x}\leq 1\},$ and for each $\vx\in\Rrr ^2$, we have that
\[
\|\vx\|_P = \max_{\va\in \mathcal{G}_P} \{\va \cdot \vx \}.
\]
Namely, let $\mathbf{b}$ be the point on $P$ where the line segment from the origin to $\vx$ intersects $P$, and let $\va\in \mathcal{G}_P$ be so that $\vb$ is part of the face of $P$ with normal $\va$. Then $\va \cdot \vb =1$, and $\|\vx\|_P =\frac{||\vx||_2}{||\mathbf{b}||_2}=\frac{\va\cdot \vx}{\va\cdot \mathbf{b}}=\va\cdot \vx$.

Next, let $\vz,\vz^*\in\Rrr^2$, and assume that $d_P(\vz,\vz^*)=\va\cdot (\vz-\vz^*)$ for $\va\in\mathcal{G}$. In other words, the distance from $\vz$ to $\vz^*$ is determined by $\va$. This means that, if $P$ is centered at $\vz$, then the line segment $\vz\vz^*$  intersects $P$ in $F(P,\va)$. Moreover, the same is true if $P$ is enlarged by a factor $M$ (as long as it remains centered at $\vz$.) Also, if $P_M$ is a version of $P$ enlarged by a factor $M$ and centered at $\vz$, and so that $\vz^*$ lies on the boundary of $P$, then  $\vz^*$ will lie on the enlarged version of $F(P_M,\va)$. Moreover, $M=d_P(\vz,\vz^*)$, and $P_M$ is the ball of radius $M$ around $\vz$. If $P$ is centered at $\vz^*$, then the line segment $\vz\vz^*$ will intersect $P$ in $F(P,-\va )$. See also Figure~\ref{fig4}.
\begin{figure} [h!]
\begin{center}
\epsfig{figure=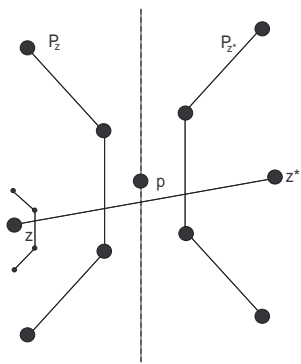,width=2in,height=2in}
\caption{Polygons $P_{\vz}$ and $P_{\vz^*}$ are similar to $P$, and centered at $\vz$, $\vz^*$, respectively. The line segment connecting $\mathbf{z}$ and $\mathbf{z}^*$ intersects $P_{
\mathbf{z}}$ and $P_{\vz^{* }}$ in parallel faces.}\label{fig4}
\end{center}
\end{figure}

For $\va\in\mathcal{G}_P$, let $F(P,\va)$ be the face of $P$ with normal $\mathbf{a},$ where $\mathbf{a}$ is pointing away from the centre of $P.$ Define $F(P,-\va)$ analogously. See Figure \ref{fig1}.

We now consider the case where the metric is smooth.
In that case,  for each point $\vp$ on the boundary of $P$, there exists a vector $\va\in\Rrr^2$ so that $\va\cdot \vp=1$. Precisely, $\va$ will be the vector perpendicular to the boundary of $P$ at $\vp$, scaled appropriately.
Now let $\mathcal{G}_P$ be a countable dense subset of such vectors in $\Rrr^2$.  In particular, we require that for each  point $\vp$ on the boundary of $P$ and all $\epsilon >0$,  exist $\va\in \mathcal{G}_P$ so that $\va\cdot \vp\geq 1-\epsilon$. In other words, the angle between $\va$ and the tangent to the boundary of $P$ at $\vp$ can be made arbitrarily small.
Let $$P^*=\{ \mathbf{x}: \mbox{for all }\mathbf{a}\in {\mathcal{G}^*}_P,\mbox{ } \mathbf{a}\cdot \mathbf{x}\leq 1\}.$$
Then the closure of $P^*$ is $P$. Moreover, we have that for all $\vx\in \Rrr^2$,
\[
\|\vx\|_P = \sup_{\va\in \mathcal{G}_P} \va\cdot \vx.
\]
Such a set $\mathcal{G}_P$ will again be called a   {\em generator set} for $P$.
Note that every convex set has a countable generator set, while only polygons have finite generator sets.

\subsection{Proof of Theorem \ref{thm:main}}

The proof of Theorem \ref{thm:main} is outlined here through a series of lemmas. The proofs of these technical lemmas can all be found in Sections \ref{sec:proofs1} and \ref{sec:proofs2}. Their proofs are based on the concept of respected lines, which we now introduce. Let
$V\subseteq \mathbb{R}^{2}$ and $f:V\rightarrow \mathbb{R}^{2}$ be an
injective map. 
Let $r\in \Rrr$,  $\va \in \Rrr^2$, and let $\ell$ be the line defined by the equation $\va \cdot \vx =r$.
The map $f$ \emph{respects} the line $\ell$ if there exists a line $\ell'$ with equation $\va ' \cdot\vx=r'$ for some $r'\in\Rrr$, $\va '\in\Rrr^2$ such that for all $\vv \in V$
\begin{equation}
\label{eq:consistent}
\mathbf{a}\cdot\vv < r \mbox{ implies that } \va ' \cdot f(\vv)<  r',\mbox{ and }\mathbf{a}\cdot \vv > r \mbox{ implies that } \va ' \cdot f(\vv )>   r'.
\end{equation}
The line $\ell'$ will be called the \emph{image} of $\ell$ under the \emph{line map}. Note that a function $f$ respects a line $\ell$ if the half-spaces on both sides of $\ell$ are mapped to half-spaces separated by another line, which can be then be considered the image of $\ell$. For example, if $f$ respects a vertical line $\ell$, then points to the left (right) of $\ell$ are mapped to the right (left) of the image of $\ell.$  Further, note that isometries respect lines. An \emph{integer parallel} of line $\ell$ is a line $\ell'$ parallel to $\ell$ so that $d(\ell,\ell')\in \Zzz$ (the distance between two lines is defined in the obvious way).

The first lemma establishes some straightforward consequences of the definitions.
\begin{lemma}
\label{lem:consistent}
Let $V$ be a countable dense set in $\Rrr^2$ and let $f:V\rightarrow V$ be a bijection. Suppose $f$ respects the line $\ell$, and $\ell'$ is an image of $\ell$ under the line map.
\begin{itemize}
\item[($i$)] If $\ell$ contains a point $\vv\in V$, then $\ell'$ contains $f(\vv)$.
\item[($ii$)] The line $\ell'$ is the unique image of $\ell$ under the line map. That is, $\ell'$ is the only line which satisfies (\ref{eq:consistent}).
\item[($iii$)] Suppose $f$ respects the line $\hat{\ell}$, and $\hat{\ell'}$ is the image of $\hat{\ell}$ under the line map. Then $\ell$ and $\hat{\ell}$ are parallel if and only if $\ell'$ and $\hat{\ell}'$ are parallel.
\item[($iv$)] If $f$ is a step-isometry and $\ell $ is a line respected by $f$, then all integer parallels of $\ell $ must also be respected.
\end{itemize}
\end{lemma}

The first step in the proof of the main result is to establish that there are some lines that must be respected by any step-isometry on a countable dense set. We consider the smooth case first.

\begin{lemma}
\label{lem:exist non-poly}
Consider $\Rrr^2$ equipped with a smooth metric $d\in\Omega$.
Let $V$ be a countable dense set in $\Rrr^2$, and let $f:V\rightarrow V$ be a step-isometry. Let $\vv_1,\vv_2\in V$, and let $\ell $ be the line through $\vv_1$ and $\vv_2$.  Then $\ell $ must be respected, and its image is the line $\ell'$ through $f(\vv_1)$ and $f(\vv_2)$. Moreover, any line through a point $\vv_3\in V$ and parallel to $\ell$ must be respected, and its image is the line through $f(\vv_3)$ and parallel to $\ell'$.
\end{lemma}

A similar lemma for the polygonal case involves a more technical argument. The proof of this lemma can be found in Section~\ref{sec:proofs2}.
\begin{lemma}
\label{lem:exist poly}
Consider $\Rrr^2$ equipped by a polygonal metric $d_P\in\Omega$.
Let $V$ be a countable dense set in $\Rrr^2$ and let $f:V\rightarrow V$ be a step-isometry. Then any line through a point  $\vv\in V$ and parallel to one of the sides of $P$ must be respected, and its image is a line through $f(\vv)$ parallel to one of the sides of $P$.
\end{lemma}

The proof of Theorem~\ref{thm:main} is based on the fact that the lines emanating from a {\em finite} number of points generate a grid which is infinitely dense, which means that a finite number of points completely determine the step-isometry. These ideas will be made precise in Lemma~\ref{lem:dense} stated below.

We introduce some terminology which describes the lines that form the dense grid of lines. Let $B$ be a set of points in  $\mathbb{R}^{2}$, and let $\mathcal{G}$ be a set of vectors in $\Rrr^2$. Define a collection of lines
$$
\mathcal{L}(B,\mathcal{G})=\bigcup\limits_{i\in \mathbb{N}}\mathcal{L}_{i}(B,\mathcal{G}).
$$
as follows. The set
$\mathcal{L}_{0}(B,\mathcal{G})$ contains all lines through points in $B$ with normal vector in $\mathcal{G}$, as well as their integer parallels. Assume that $\mathcal{L}_{i}(B,\mathcal{G})$ has been defined for some $i\geq 0.$ Then $\mathcal{L}_{i+1}(B,\mathcal{G})$ consists of all lines with normal vector in $\mathcal{G}$ going through a point $\vp$ which is an intersection point of two lines in $\mathcal{L}_{i}(B)$.

Now let $V$ be a countable dense set in $\Rrr^2$, and $f:V\rightarrow V$ a step-isometry.
Lemmas~\ref{lem:exist non-poly} and \ref{lem:exist poly} can be used to show that, given a set $B\subseteq V$, there exists a countable generator set $\mathcal{G}$ of $P$ so that $\mathcal{L}_0(B,\mathcal{G})$ must be respected by $f$. Using an inductive argument, it can be shown that, if all lines in $\mathcal{L}_0(B,\mathcal{G})$ must be respected, then  all lines in $\mathcal{L}(B,\mathcal{G})$ must be respected. This results in the following crucial lemma.

\begin{lemma}\label{lem:respect}
Let $d_P\in\Omega$, let $V$ be a countable dense set in $\Rrr^2$, and let $f:V\rightarrow V$ be a step-isometry. Then
\begin{enumerate}
\item[(a)] there exists a countable generator set $\mathcal{G}$ for $P$, so that for each set $B\subseteq V$, $f$ must respect all lines in $\mathcal{L}(B,\mathcal{G})$,

\item[(b)] there exists a map $\sigma:\mathcal{G}\rightarrow \Rrr^2$ 
such that, if  a line  $\ell\in \mathcal{L}(B,\mathcal{G})$ has normal vector $\va\in\mathcal{G}$ then its image under the line map has normal vector  $\sigma(\va)$, and
\item[(c)] the set $\mathcal{G'}=\{ \sigma(\va ):\va\in\mathcal{G}\}$ is a generator set for $P$.
\end{enumerate}
\end{lemma}

If the vectors in $\mathcal{G}$ fall into two parallel classes, then $\mathcal{L}(B,\mathcal{G})=\mathcal{L}_0(B,\mathcal{G})$. In particular, in this case no new lines are generated in the induction step. The following lemma shows that, if $d_P$ is not a box metric, and thus, any generator set $\mathcal{G}$ for $P$ contains at least three vectors that are pairwise non-parallel, then a finite set $B$ can be chosen so that lines in $\mathcal{L}(B,\mathcal{G})$ are arbitrarily close together.

\begin{lemma}\label{lem:dense}
Let $d_P\in \Omega$, and let $\mathcal{G}$ be a countable generator set for $P$ which contains at least  three vectors that are pairwise non-parallel. Let $B=\{\vp,\mathbf{q} \},$ and fix $\va\in \mathcal{G}$, where $\va \cdot (\mathbf{p}-\vq)=r\in (0,1)$.
Then the family $\mathcal{L}(B,\mathcal{G})$ contains all the following lines,
where  $z_1,z_2\in \mathbb{Z}$:%
\[
\mathbf{a}\cdot (\mathbf{x}-\vq)=z_1r+z_2.
\]
\end{lemma}

Thus, if $B=\{ \vp,\vq\}$ is so that  $d_P(\vp,\vq)<\epsilon$, then for all vectors $\va \in\mathcal{G}$,   $\va \cdot (\mathbf{p}-\vq)<\epsilon$, and the lines in $\mathcal{L}(B,\mathcal{G})$ generate a grid where lines are at most $\epsilon $ apart.
If $B\subseteq V$, then by Lemma \ref{lem:consistent}, any step-isometry must be consistent with this grid.  




The final lemma leads to the proof of the main result. It shows that any step-isometry on a countable dense set gives rise to an isometry on a dense set of points in $\Rrr^2$ formed by a grid of lines which all have to be respected.



With these lemmas at our disposal, we may now supply a proof of our first main result.

\begin{proof}[Proof of Theorem~\ref{thm:main}]
For the reverse direction, let $d_P\in \Omega$ be a norm-based metric which is not a box metric, and let 
$V$ be any dense set in $\Rrr^2$. Let $f:V\rightarrow V$ be a step-isometry; we must show that $f$ is in fact an isometry.   Let $\mathcal{G}$ be a countable generator set for $P$ for which the conclusions of Lemma \ref{lem:respect} (a) hold. Let $\sigma:\mathcal{G}_P\rightarrow \Rrr^2$ be as given by Lemma \ref{lem:respect} (b).

Fix $\vp,\vq\in V$, and let $B=\{ \vp,\vq\}$. 
Fix $\va\in\mathcal{G}$, and let $r=\va\cdot (\vq-\vp)$ and $r'=\sigma(\va)\cdot (f(\vq)-f(\vp))$.
By Lemma \ref{lem:dense}, for any $z_1,z_2\in \Zzz$ the line $\ell$ with equation $\va\cdot (\vx -\vq)= z_1r+z_2$ is in $\mathcal{L}(B,\mathcal{G})$, and by Lemma \ref{lem:consistent}, this line must be respected.

Following the proof of Lemma~\ref{lem:dense}, it can be easily deduced that the image of $\ell$ must have as its  equation  $\sigma(\va)\cdot (\vx -\vq')= z_1r'+z_2$, where $\vq'=f(\vq)$. Namely, each of the points $\vp_i$ used in the proof and in Figure~\ref{fig6} is defined by the intersection of previously given lines. Their image $f^*(\vp_i)$ will be similarly defined by the intersection of the images of those lines. Thus, the images again follow a layout as in Figure~\ref{fig6} (right), but in this case the distance between the images of the reference points $\vq$ and $\vp$ equals $r'$.

We now claim that $r=r'$. Assume that this is not the case. Then there exist integers $z_1,z_2\in \Zzz$ so that $z_1r+z_2<1$ and $z_1 r'+z_2 >1$. Choose a point $\vv\in V$ so that $z_1r+z_2<\va\cdot (\vv-\vq) < 1$. Thus, $\vv$ lies to the right of the line with equation $\va\cdot(\vx-\vq)=z_1r+z_2$, and to the left of the line with equation $\va\cdot (\vx-\vq) =1$.  Thus, its image $\vv'=f(\vv)$ must lie to the right of the line with equation $\sigma(\va)\cdot (\vx-\vq')=z_1r'+z_2$, and to the left of the line with equation $\sigma(\va)\cdot (\vx-\vq')=1$. Hence, $z_1r'+z_2< \sigma(\va)\cdot (\vv-\vq)<1$, which is a contradiction.



Since $\va$ was arbitrary, we conclude that for all $\va\in\mathcal{G}$:
\begin{equation}
\va\cdot (\vp-\vq)= \sigma(\va)\cdot(\vp'-\vq'). \label{ooo}
\end{equation}
Let $\mathcal{G'}=\{ \sigma(\va ):\va\in\mathcal{G}\}$. By Lemma \ref{lem:respect} (c),  $\mathcal{G}'$ is again a generator set for $P$. By the definition of a generator set, we have that

\begin{eqnarray*}
d_P(\vp,\vq)&=&\sup\{ |\va\cdot (\vq-\vp)|: \va\in\mathcal{G}\}\\
&=&\sup\{ |\sigma(\va)\cdot (\vq'-\vp')|:\va\in\mathcal{G}\}\\
&=&\sup\{ |\va'\cdot (\vq'-\vp')|:\va'\in\mathcal{G'}\}\\
&=&d_P(\vq',\vp'),
\end{eqnarray*}
where the second equality follows by (\ref{ooo}). Since $\vp,\vq$ were arbitrary, $f$ is an isometry. This shows that $d_P$ has the truncating property, and thus, concludes the first part of the proof.

\smallskip
For the forward direction, assume that $d_P\in \Omega$ is a box metric. Thus, $P$ is a parallelogram, and $d_P$ has a set of two non-parallel generators $\va_1$ and $\va_2$, so that
\[
P=\{ \vx : -1\leq \va_i\cdot\vx \leq 1\mbox{ for }i=1,2\}.
\]

Define a set $A\subseteq \Rrr$ to be {\sl integer distance free (idf)} if for distinct  $x,y\in  A$, $|x-y|$ is not an integer. In \cite{bon2}, it was shown that, for any  {\sl idf} set $A$ in $\Rrr$, a function $f:A\rightarrow A$ is a step-isometry if it has the property that, for any $x,y\in A$, we have that $x-\lfloor x\rfloor <y-\lfloor y\rfloor$ if and only if $f(x)-\lfloor f(x)\rfloor <f(y)-\lfloor f(y)\rfloor$. It is straightforward to see that if $A$ is dense, then it is possible to construct functions $f$ that have this property but are not isometries (see \cite{bon2} for details).

Let $V$ be a dense set in $\Rrr^2$, and let $V_i$, where $i=1,2$, be defined as
\[
V_i=\{ \va_i\cdot \vv:\vv\in V\}.
\]
Let $V$ be such that both $V_1$ and $V_2$ are  {\sl idf}, and let $f_1$, $f_2$ be step isometries on $V_1$ and $V_2$, respectively, which are not isometries. Define $f:V\rightarrow V$ so that for each $\vv\in V$ so that $\vv=x_1\va_1+x_2\va_2$, $f(\vv )=f(x_1)\va_1+f(x_2)\va_2$. The function $f$ is a step-isometry which is not an isometry.

Finally, we show that there exists a countable dense subset $V\subseteq \Rrr^2$ with the desired property. Namely, take any  countable dense subset $V^*$ which contains $\mathbf{0}$. Let $A=\{ \va_i\cdot \vv:i=1,2\}$. Choose $\alpha >0$ so that $\alpha \mathbf{q}\not\in\Zzz$ for all $\mathbf{q}\in A$. Since $A$ is countable, such $\alpha $ exists. Now let $V=\{ \alpha \vv:\vv\in V^*\}$. We therefore, have that box metrics do not have the truncating property. \end{proof}

\subsection{Proof of Theorem \ref{non2}}
We finish with the proof of our second main result. For graphs $G$ and $H$, a \emph{partial isomorphism} from $G$ to $H$ is
an isomorphism of some finite induced subgraph of $G$ to an induced subgraph of $H$. A standard approach to show two countably infinite graphs are isomorphic is to build a chain of
partial isomorphisms whose union gives an isomorphism. Using the probabilistic method, we show how this fails for random geometric graphs whose metric is not a box metric. In the following proof, the probability of an event $A$ is denoted by $\mathbb{P}(A).$

\begin{proof}[Proof of Theorem~\ref{non2}]
Suppose next that $\Rrr^2$ is equipped with a metric $d\in\Omega$ which is not box, and let $P$ be the shape of $d$. Let $\mathcal{G}$ be a countable generator set for $P$ which contains at least  three vectors that are pairwise non-parallel, and let $V$ be a countable dense set. Consider graphs produced by
the model $\mathrm{LARG}(V,p)$ for some probability $p\in (0,1)$.

Define an enumeration $\{ v_i:i\in \Nnn^{+}\}$ of $V$ to be \emph{good} if $v_i\not=v_j$ for $i\not=j$,
$d(v_i,v_{i+1})< 1$ for all $i\in \Nnn^+$ and for all $i\geq 4$, there exist integers $j,k,\ell$, $0<j<k<\ell <j$, so that the distance of $v_i$ to $v_j,v_k,v_{\ell}$ is determined by three different vectors in $\mathcal{G}$. As discussed earlier, the last condition implies that, if the positions of $v_1,\dots ,v_i$ are given, then the position of $v_j$ is completely determined by the distances from $v_i$ to $v_j,v_k,v_{\ell}$.

\begin{claim}\label{claim:good}
Any countable dense set $V$ has a good enumeration.
\end{claim}
\begin{proof}
For a positive integer $n$, we call $\{ v_i:1\le i \le
n\}$ a \emph{partial good enumeration} of $V.$ We prove the claim
by constructing a chain of partial good enumerations
by induction. Since $V$ is  dense in $\Rrr^2$, we may choose a triangular set whose points are pairwise within
$1$ of each other (see also the discussion in Section \ref{subsec:main}). Let $V_1 =\{v_1,v_2,v_3\}.$
Enumerate $V \setminus
V_1$ as $\{u_i: i\ge 2\}.$ Starting from $V_1$, we
inductively construct a chain of partial good enumerations $V_n$,
$n\geq 1$, so that for $n\geq 2$, $V_n$ contains $\{u_i: 2\le i
\le n\}.$

We now want to form $V_{n+1}$ by adding $u=u_{n+1}$. If $u\in V_n$,
then let $V_{n+1}=V_n$. Assume without loss of generality that
$u \not\in V_n.$ Let $N=|V_n|$. If $d(v_N ,u) < 1$ and the position of $u$ is determined by $V_n$,
then let $v_{N+1}=u$ and add it to $V_n$ to form $V_{n+1}.$ Otherwise, by the
density of $V,$ choose a finite path $P=p_0,\dots, p_\ell$ of
points of $V\setminus V_n$
starting at $v_N=p_0$ and ending at $u=p_\ell$ so that two consecutive points
in the path are distance at most 1, and their position is determined by their distances to points in $V_n$. Then add the vertices of
$P$ to $V_n$ to form $V_{n+1}$ and enumerate them so that $v_{N+i}=p_i$ for
$i=0,1,\dots,\ell$.
Taking the limit of this chain,
$\bigcup_{n\ge 1} V_n$ is a
good enumeration of $V$, which proves the claim. \end{proof}

Let $V=\{ v_i:i\geq 1\}$ be a good sequence in
$V,$ and for any $n$, let $V_n = \{ v_i: 1\le i \le n \}.$ Let $G$ and $H$ be two graphs produced by $\mathrm{LARG}(V,p)$. We say that two pairs $\{ v,w\}$ and $\{ v',w'\}$ of vertices
are \emph{compatible} if $\{ v,w\}$ are adjacent in $G$ and $\{
v',w'\}$ are adjacent in $H$ or $\{ v,w\}$ are non-adjacent in
$G$ and $\{ v',w'\}$ are non-adjacent in $H$. For two pairs $\{ v,w\}$
and $\{ v',w'\}$ such that $d(v,w)=d(v',w')$, the probability that they
are compatible equals
\[
p^*=\left\{\begin{array}{ll}
p^2+(1-p)^2 & \text{if }d(v,w)<\delta\text{, and}\\
1 &\text{otherwise.}
\end{array}\right.
\]

Suppose that $G$ and $H$ are isomorphic, and let $f$ be an isomorphism. By Lemma~\ref{newlemmm}, $f$ is a step-isometry on $V$.  By Theorem \ref{thm:main}, $f$ must be an isometry. Since we are following a good enumeration, the images of
a triangular set $V_1 = \{ v_1, v_2, v_3 \}$ in $\Rrr^2$ (we identify these with vertices so do not denote them in bold) determine $f$ completely.  Let $A_n$ be the event that there exists a partial
isomorphism $f$ from the subgraph induced by $V_n$ into $H$ so that
$f(V_1)\subseteq V_n$, and let
$$
A^*_n=\bigcap_{\nu \geq
  n} A_{\nu}.
$$
Note that $A^*_{n}\subseteq A^*_{n+1}$ for all $n$.

Next, we estimate the probability of $A^*_n$. Note first that
$\mathbb{P}(A^*_n)\leq \mathbb{P}(A_\nu)$ for all $\nu\geq n$. For any tuple
$(u_1,u_2,u_3)$ of distinct vertices in $V_n$, let
$C_n(u_1,u_2,u_3)$  be the event that there exists a partial
isomorphism $f$ from the subgraph induced by $V_n$ in $G$ to $H$ so that $f(v_i)=u_i$ for
$i=1,2,3$. If $C_n$ happens, then all pairs
$(v_i,v_{i+1})$ and $(f(v_i),f(v_{i+1}))$ must be compatible, for
$1\leq i<n$. Therefore, $$\mathbb{P}(C_n(u_1,u_2,u_3))\leq (p^*)^{n-1}.$$ Now
\[
A_n=\bigcup_{\{u_1, u_2, u_3\}\subseteq V_n} C_n(u_1,u_2,u_3),
\]
so for $n\ge 3$ we have that $\mathbb{P}(A_n)\leq n^{2k+2} (p^*)^{n-1}$, and
\[
\mathbb{P}(A^*_n)\leq \inf\{ \nu^{2k+2} (p^*)^{\nu-1} : \nu\geq n\} =0.
\]

If $B$ is the event that $G$ and $H$ are
isomorphic, then $$B\subseteq \bigcup_{n\in \Nnn^+} A^*_n.$$ Since
the union of countably many sets of measure zero has measure zero, we
conclude that $\mathbb{P}(B)=0$, and thus, with probability $1,$ $G$ is not isomorphic to $H.$ This shows that $d$ is unpredictable, and completes the first part of the proof.

\smallskip
For the  suppose that $\Rrr^2$ is equipped with a box metric $d_P\in\Omega$. As in the proof of Theorem \ref{thm:main}, the shape $P$ can be described as
\[
P=\{ \vx : -1\leq \va_i\cdot\vx \leq 1\mbox{ for }i=1,2\}.
\]
and let $T:\mathbb{R}^2\rightarrow \mathbb{R}^2$  be the invertible linear transformation which sends $\va_1$ to $(1,0)^T$ and $\va_2$ to $(0,1)^T$. Let $V$ be a set dense in $\Rrr^2$, and let $T(V)=\{ T(\vx):\vx\in V\}$. Note that $T(V)$ is also dense in $\Rrr^2$. Then for any $p\in (0,1)$ the model $\mathrm{LARG}(V,p)$ in $\Rrr_2$ equipped with the box metric is equal to the model $\mathrm{LARG}(T(V), p)$ in $\Rrr^2$ with the $L_\infty$ metric.
It was shown in  \cite{fu} Theorem~14 that, if $V$ is a dense set with the additional requirement that no two vertices have integer distance, then with probability 1 any two graphs produced by $\mathrm{LARG}(V,d)$ are isomorphic.  It is straightforward to see there exist dense sets $V$ so that for all $\vx,\vy\in V$, $\va_i\cdot (\vx-\vy)\not\in\Zzz$. For such sets $V$, $T(V)$ has the property that there is non-integer distance between any $x$-coordinates and any $y$-coordinates of points in $T(V)$. We therefore, have that $d_P$ is predictable. \end{proof}

\section{Proofs of Lemmas~\ref{lem:consistent},~\ref{lem:exist non-poly},~\ref{lem:respect}, and \ref{lem:dense} }\label{sec:proofs1}

\begin{proof}[Proof of Lemma \ref{lem:consistent}] Suppose $f$ respects the line $\ell$, and $\ell'$ is the image of $\ell$ under the line map.  Property ($i$) follows directly from (\ref{eq:consistent}): if $\ell$ contains a point $\vv\in V$, then  $\ell'$ must contain $v$.

Suppose, by contradiction, that ($ii$) is false, so assume that there exists a second image of $\ell$ under the line map. Precisely, assume there exists a line $\ell^*$ so that if a point $\vv\in V$ is to the left (right) of $\ell$, then it is to the left (right) of both $\ell'$ and $\ell^*$. However, this implies that the region of points to the right of $\ell$ and to the left of $\ell^*$ does not contain any images under $f$ of points in $V$. This contradicts the fact that $V$ is dense in $\Rrr^2$, and $f$ is a bijection.

To prove $(iii)$, assume that $f$ respects the line $\hat{\ell}$, and $\hat{\ell'}$ is the image of $\hat{\ell}$ under the line map, and suppose that $\ell$ and $\hat{\ell}$ are not parallel. Then $\ell$ and $\hat{\ell}$ divide the plane into four regions, consisting of points that are to the left or right of $\ell$ and to the left or right of $\hat{\ell}$. Now suppose, by contradiction, that $\ell'$ and $\hat{\ell}'$ are parallel. Then $\ell'$  and $\hat{\ell}'$ divide the plane into three regions. Thus, there is one combination, for example, to the left of $\ell'$ and to the right of $\hat{\ell'}$, which is an impossibility. Since $V$ is dense in $\Rrr^2$, all four regions formed by $\ell$ and $\hat{\ell}$ contain points of $V$, which gives a contradiction. To prove the converse, apply the analogous argument to $f^{-1}$.

For $(iv)$, suppose that $f$ is a step-isometry and $\ell $ is a line respected by $f$, with image $\ell'$. Let $\hat{\ell}$ be an integer parallel of $\ell$. Let $\va\cdot\vx=r$ be the equation defining $\ell$, and $\va\cdot \vx = r+z$ the equation defining $\hat{\ell}$, where $z\in \Zzz$. Without loss of generality, assume $z>0$.  Let $\va'\cdot \vx=r'$ be the equation of $\ell'$. We claim that the line $\hat{\ell}$ must be respected, and its image under $f$ is the line $\hat{\ell'}$ with equation $\va'\cdot\vx=r'+z$.

Suppose, by contradiction, that there exists a point $\vw\in V$ with image $\vw'=f(\vw)$
so that $\va\cdot\vw< r+z$, while  $\va'\cdot\vw' >r'+z$. Now choose $\uu\in V$ so that $d_P(\uu,\vw)<z$ and $\va\cdot\uu <r$. Since $V$ is dense in $\Rrr^2$ and the distance between $\ell $ and $\hat{\ell}$ equals $z$, we can choose such a $\uu$. Let $\uu'=f(\uu)$. Since $f$ respects $\ell$, we have that $\va'\cdot \uu'<r'$. Since $\vw$ and $\uu$ lie on opposite sides of $\ell'$ and $\hat{\ell'}$, we conclude that  $d_P(\uu',\vw')>d_P(\ell',\hat{\ell'})=z$. This contradicts the fact that $f$ is a step-isometry. See Figure~\ref{fignew1}. \end{proof}
\begin{figure}[h!]
\begin{center}
\epsfig{figure=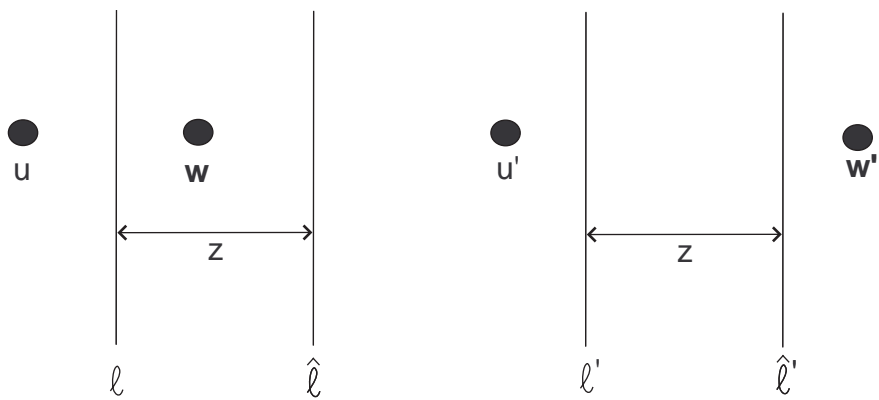}
\caption{The distance between $\vu'$ and $\vw'$ must be greater than $z$, while the distance between $\vu$ and $\vw$ is less than $z$. This leads to a contradiction.}\label{fignew1}
\end{center}
\end{figure}

\begin{proof}[Proof of Lemma \ref{lem:exist non-poly}]

Let $\ell $ be the line through $\mathbf{v}_{1}$ and $\mathbf{v}_{2},$ and $%
\ell ^{\prime }$ be the line through $\mathbf{v}_{1}^{\prime }$ and $\mathbf{%
v}_{2}^{\prime }.$ Suppose, by contradiction, that $\mathbf{u}$ and $\mathbf{w}$ are both to the right of $\ell
$, but $\mathbf{u}'$ is to the
left of $\ell^{\prime}$ and $\mathbf{w}'$ is to the right of $\ell^{\prime} $. The proof is based on the following fact: Since $\vu$ and $\vw$ are on the same side of $\ell$, no line which is
\textquotedblleft almost parallel\textquotedblright\ to $\ell $
and close to $\ell $ can separate $\mathbf{u}$ from $%
\mathbf{w}$. More precisely, for all $\gamma >0$ there exists $M
$ large so that if two points $\mathbf{x}$ and $\mathbf{y}$ are located at distance less than $\gamma $ from $\ell $ and distance greater than $M$ from $\{ \vv_1,\vv_2\}$, and $\vx$ and $\vy$ are located on opposite sides of $\{ \vv_1,\vv_2\}$, then
any line which separates $\{\mathbf{v}_{1},\mathbf{v}_{2}\}$ from $\{\mathbf{x},\mathbf{y}\}$ must have such a small angle with $\ell$ that it cannot separate $\mathbf{u}$
from $\mathbf{w}.$ See Figure~\ref{newfig11}. This fact is obviously true for the Euclidean
metric, and since $d=d_{P}$ is bounded by a constant multiple of the
Euclidean metric, it holds also for $d.$
\begin{figure}[h!]
\begin{center}
\epsfig{figure=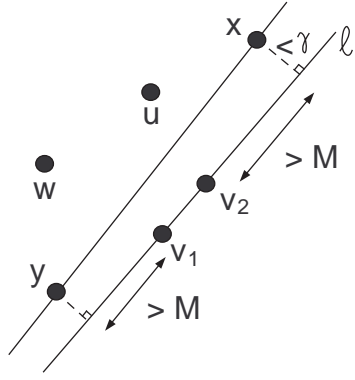}
\caption{The points $\mathbf{u}$ and $\mathbf{w}$ cannot be separated.}\label{newfig11}
\end{center}
\end{figure}

The second fact used in the proof is that $P$ (the shape of $d)$ when
sufficiently enlarged, will locally look like a line (Note this holds only
for smooth metrics.) Finally, we use the fact that any point $\mathbf{p}$ for which%
\[
d(\mathbf{p},\mathbf{v}_{1})=d(\mathbf{p},\mathbf{v}_{2})+d(\mathbf{v}_{1},\mathbf{v}_{2})
\]
must lie on the line $\ell .$ (Once again, this holds for smooth
metrics.) Thus, by controlling the distances from a point to $\mathbf{v}_{1}
$ and $\mathbf{v}_{2}$ we can guarantee that the point lies close to $\ell$.
\smallskip

\noindent \textbf{Claim.}
If $\mathbf{x}$ is a point such that%
\[
\left\lfloor d(\mathbf{x},\mathbf{v}_{1})\right\rfloor
=\left\lfloor d(\mathbf{x},\mathbf{v}_{2})\right\rfloor
+\left\lfloor d(\mathbf{v}_{1},\mathbf{v}_{2})\right\rfloor ,
\]
then $d(\mathbf{x},\ell )\leq 1.$
\smallskip

\begin{proof}[Proof of Claim] Let $M=\left\lfloor d(\mathbf{x},\mathbf{v}_{2})\right\rfloor$. Let $\mathbf{p}$ be a point on $\ell
$ such that $d(\mathbf{p},\mathbf{v}_{2})=M$ and $d(\mathbf{p},%
\mathbf{v}_{1})=M+d(\mathbf{v}_{1},\mathbf{v}_{2}).$ By
the triangle inequality, we have that%
\[
d(\mathbf{x},\mathbf{v}_{1})+d(\mathbf{x},\mathbf{p})\leq d(%
\mathbf{p},\mathbf{v}_{1})=M+d(\mathbf{v}_{1},\mathbf{v}%
_{2}).
\]

Hence,
\begin{eqnarray*}
d(\mathbf{x},\vp) &\leq &M+d(\mathbf{v}_{1},\mathbf{v}_{2})-d(\mathbf{x},\mathbf{v}_{1}) \\
&\leq &M+\left\lfloor d(\mathbf{v}_{1},\mathbf{v}_{2})\right\rfloor -\left\lfloor d(\mathbf{x},\mathbf{v}_{1})\right\rfloor +1 \\
&=&M-\left\lfloor d(\mathbf{x},\mathbf{v}_{2})\right\rfloor +1 \\
&=&1.
\end{eqnarray*}

We therefore, have that $d(\mathbf{x},\ell )\leq d(\mathbf{x},%
\mathbf{p})\leq 1.$ \end{proof}

We now complete the proof of the lemma. Define two points $\mathbf{x}$ and $\mathbf{y}$ to be \emph{antipodal} to $\{ \mathbf{v}_1 ,\mathbf{v}_2 \}$ if

\begin{eqnarray*}
\left\lfloor d(\mathbf{x},\mathbf{v}_{1})\right\rfloor -\left\lfloor d(%
\mathbf{x},\mathbf{v}_{2})\right\rfloor  &=&\left\lfloor d(\mathbf{v}_{1},%
\mathbf{v}_{2})\right\rfloor , \\
\left\lfloor d(\vy,\mathbf{v}_{1})\right\rfloor -\left\lfloor d(\vy,\mathbf{v}%
_{2})\right\rfloor  &=&\left\lfloor d(\mathbf{v}_{1},\mathbf{v}_{2})\right\rfloor ,%
\text{ }
\end{eqnarray*}

and

\[
\left\lfloor d(\mathbf{x},\mathbf{y})\right\rfloor \geq \left\lfloor d(%
\mathbf{x},\mathbf{v}_{1})\right\rfloor +\left\lfloor d(\mathbf{v}_{1},%
\mathbf{v}_{2})\right\rfloor +\left\lfloor d(\vy,\mathbf{v}_{2})\right\rfloor .
\]

By the Claim, antipodal points $\mathbf{x},\mathbf{y}$ satisfy $d(\mathbf{x}%
,\ell ),d(\mathbf{y},\ell )\leq 1$. 
Choose $M$ so that, for any two points $%
\mathbf{x}$ and $\vy$ antipodal to $\{\mathbf{v}%
_{1},\mathbf{v}_{2}\},$ if an enlargement by a factor $M$
of $P$ separates $\{\mathbf{v}_{1},\mathbf{v}_{2}\}$
from $\{\mathbf{x},\mathbf{y}\},$ then it cannot separate
$\mathbf{u}$ from $\mathbf{v}.$ This can be
done by making sure that the line through $\mathbf{x}$ and $%
\mathbf{y}$ is almost parallel to $\ell $ and does not
separate $\mathbf{u}$ from $\mathbf{v}$, and then making $M$ large enough so that any polygon separating $\{\mathbf{x},\mathbf{y}\}$ from $\{\mathbf{v}_1,\mathbf{v}_{2}\}$ must have a boundary between $\mathbf{x}$ and $%
\vy$ that is \textquotedblleft almost straight\textquotedblright .

Now choose points $\vx'$, $\vy'$ and $\vp'$ in $V$ so that $\vx'$ and $\vy'$ are antipodal to $\{\mathbf{v}'_{1},\mathbf{v}'_{2}\}$, and the following inequalities hold for their distances:
\begin{eqnarray*}
d(\mathbf{p}',\mathbf{x}'),d(\mathbf{p}',\mathbf{y}') &<&M, \\
d(\mathbf{p}',\mathbf{v}'_{1}),d(\mathbf{p}',\mathbf{v}'_{2}) &>&M,\text{ }
\end{eqnarray*}

and

\[
d(\mathbf{p}',\mathbf{u}')<M,\, d(\mathbf{p}',\mathbf{w}')>M.
\]

Since $\mathbf{u}'$ and $\mathbf{w}'$ are on opposite sides of $\ell' $ and $V$
is dense, $\vp'$ can be found. Thus, the ball of size $M$ centred at $\vp'$
separates $\{\mathbf{x}',\mathbf{y}'\}$ from $\{\mathbf{v}'_{1},\mathbf{v}'%
_{2}\},$ and $\mathbf{u}'$ from $\mathbf{v}'.$

Now since $f$ is a step-isometry, so is its inverse. Therefore, the same inequalities must hold for the
pre-images $\mathbf{x}=f^{-1}(\vx' ),$ $\mathbf{y}=f^{-1}(\vy')$ and $\mathbf{p}=f^{-1}(\vp')$.
Also, $\mathbf{x}$ and $\mathbf{y}$ are antipodal
points with respect to $\{\mathbf{v}_{1},\mathbf{v}_{2}$. The ball of radius $M$ centered at $\mathbf{p}$ separates $\{%
\mathbf{x},\mathbf{y}\}$ from $\{\mathbf{v}_{1}
,\mathbf{v}_{2}\},$ and separates $\mathbf{u}$ from $\mathbf{v}$. By our choice of $M,$ this gives a contradiction.
\end{proof}

We defer the proof of Lemma~\ref{lem:exist poly} to Section~\ref{sec:proofs1}.

\begin{proof}[Proof of Lemma~\ref{lem:respect}]
If $d_P$ is a polygonal metric, then let $\mathcal{G}$ be the (finite) set of generators of $P$; that is, the set of normals to the sides of $P$. If $d_P$ is a smooth metric, then let $\mathcal{G}$ be the set of generators $\va$ of $P$ such that there is a line through two points of $V$ which has $\va$ as its normal vector.
Since $V$ is dense, $\mathcal{G}$ will be a generator set of $P$, and since $V$ is countable, $\mathcal{G}$ is also countable.
We will show, using induction, that for each set $B\subseteq V$,
$f$ must respect all lines in $\mathcal{L}(B,\mathcal{G})$.

Lemmas \ref{lem:exist non-poly} and \ref{lem:exist poly} show that any line through a point in $B$ and with normal in $\mathcal{G}$ must be respected. Lemma \ref{lem:consistent} ($iv$), shows that any integer parallel of these lines must also be respected. This shows that $f$ must respect all lines in $\mathcal{L}_0(B,\mathcal{G})$, and thus, establishes the base case of the induction.

Define $\sigma: \mathcal{G}\rightarrow \Rrr^2$ as follows. First, let $d$ be a polygonal metric. Fix $\va\in\mathcal{G}$, and then let $\ell\in \mathcal{L}_0(B,\mathcal{G})$ be a line through a point $\vb\in B$ with normal vector $\va$, so $\ell $ has  equation $\va\cdot \vx=\va\cdot \vb$. Let $\ell'$ be the image of $\ell$ under the line map. Then $\sigma (\va)$ is the normal of $\ell'$. Specifically, $\ell'$ has as its equation $\sigma(\va)\cdot \vx =\sigma(\va)\cdot f(\vb)$. By Lemma \ref{lem:consistent}, item ($iii$), parallel lines that must be respected have images that are also parallel. Thus, any line which must be respected by $f$ and which has normal vector $\va$, must have as its image a line with normal vector $\sigma(\va)$. By item ($ii$) of the same lemma, if $\ell$ contains point $\vb$, then its image contains $f(\vb)$. This shows that $\sigma$ is well-defined.

Similarly, we define $\sigma $ for the case where $d$ is a smooth metric. Let $\va\in\mathcal{G}$ be the normal to the line through points $ \vv_1,\vv_2\in V$. Then let $\sigma(\va)$ be the normal of the line through $f(\vv_1)$ and $f(\vv_2)$. By Lemma \ref{lem:exist non-poly},  any line  $\ell\in\mathcal{L}_0(B,\mathcal{G})$ through a point $\vb\in B$ with normal vector $\va$ will have as its image under the line map the line with equation $\sigma(\va)\cdot \vx =\sigma(\va)\cdot f(\vb)$.

Consider the set $\mathcal{G'}=\{ \sigma(\va):\va\in\mathcal{G}\}$. Assume that all vectors $\sigma (\va)$ are scaled so that $|\sigma(\va)\cdot \vp|\leq 1$ for all points on the boundary of $P$, and equality is achieved. By Lemma \ref{lem:exist poly}, if $d_P$ is a polygonal metric, all lines parallel to one of the sides of $P$ are mapped to lines that are again parallel to the sides of $P$, and thus, $\mathcal{G}'=\mathcal{G}$. If $d_P$ is a smooth metric, then $\mathcal{G}$ is a countable dense subset of the set of all vectors from the origin to a point on the boundary of $P$. It is easy to see that $\mathcal{G'}$ must also be dense. Assume, to the contrary, that there exist two vectors $\va,\va'$ so that no vector in $\mathcal{G'}$ is \lq\lq between" $\sigma(\va)$ and $\sigma(\va')$. Fix $\vb\in B$. Let $\ell$ and $\ell'$ be the lines through $\vb$ with normal vectors $\va$ and $\va'$. Then, since $\mathcal{G}$ is dense, there must be a line $\ell^*$ through $\vb$ which separates $\ell$ from $\ell'$. This line must have a slope between the slopes of $\va$ and $\va'$. Thus, the image  of $\ell^*$ must separate the images of $\ell$ and $\ell'$ and have a slope between that of $\sigma(\va)$ and $\sigma(\va')$, which contradicts our assumption. Therefore, whether $d_P$ is polygonal or smooth, in both cases $\mathcal{G}'$ is a generator set for $P$.

We now proceed with the inductive step.
Precisely, we will recursively define a function $f^*:W\rightarrow \Rrr$, where $W$ is the set of all intersection points of lines in $\mathcal{L}(B,\mathcal G)$, and show that, for any  line $\ell\in \mathcal{L}(B,\mathcal{G})$ with equation $\va\cdot \vx = \va\cdot \vw$ ($\va\in\mathcal{G}$, $\vw\in W$), its image $\ell'$ is the line with equation $\sigma(\va)\cdot \vx = \sigma(\va)\cdot f^*(\vw) $.

To simplify notation, we will use $\mathcal{L}_n$ to denote $\mathcal{L}_n(B,\mathcal G)$.
For all $k\geq 0$, let $$\mathcal{L}_{\leq k}=\bigcup_{i=1}^k\mathcal{L}_i,$$ and let $W_k$ be the set of all intersection points of lines in  $\mathcal{L}_{\leq k}$. For the inductive step,
assume the statement holds for a fixed $k\geq 0$. Precisely, assume that $f$ must respect all lines in $\mathcal{L}_{\leq k}$ and  that  $f^*:W_k\rightarrow \Rrr$ is defined so that
for any  line $\ell\in \mathcal{L}_{\leq k}$ with equation $\va\cdot \vx = \va\cdot \vw$ ($\va\in\mathcal{G}$, $\vw\in W_k$), its image $\ell'$ is the line with equation $\sigma(\va)\cdot \vx = \sigma(\va)\cdot f^*(\vw) $.
We will show that $f^*$ can be extended so that the same holds for $\ell\in \mathcal{L}_{k+1}$.

We extend $f^*$ to $W_{k+1}$ by taking the image under $f^*$ of the intersection of two lines in $\mathcal{L}_{\leq k}$ to be the  intersection of the images of the lines under the line map. Let $\va_1,\va_2,\va_3\in \mathcal{G}$ be such that no two are linearly dependent.
Suppose that $\ell_1, \ell_2 \in\mathcal{L}_{\leq k}$.  Let $\ell_3$ be a line in $\mathcal{L}_{k+1}\setminus \mathcal{L}_{\leq k}$ formed by the intersection of $\ell_1$ and $\ell_2.$ Moreover, assume $\ell_i$ has normal vector $\va_i$, for $i=1,2,3$.  In this proof, we will say that a point $\mathbf{p}$ is to the right (left) of a line with equation $\mathbf{a}\cdot \mathbf{x}=t$ if $\mathbf{a}\cdot\mathbf{p}>t$ ($\mathbf{a}\cdot\mathbf{p}<t$). See Figure \ref{induct} for a visualization of the proof.

\begin{figure} [h!]
\begin{center}
\epsfig{figure=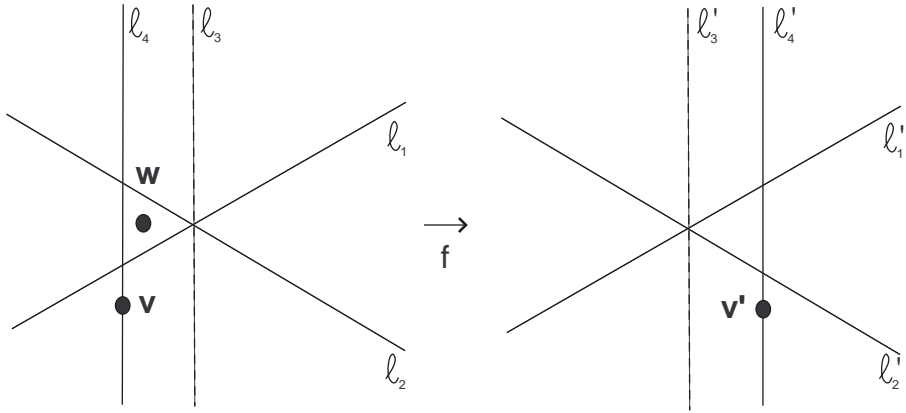}
\caption{The point $f(\mathbf{w})$ does not exist.}\label{induct}
\end{center}
\end{figure}

Assume by contradiction that $f$ does not respect $\ell_3$. Precisely, assume there exists $\mathbf{v}\in V$ which is to the left of $\ell_3$ such that $\vv'=f(\vv)$ is to the right of the line $\ell_3'$.   Let $\ell_4$ be the unique line parallel to $\ell _3$ and through $\mathbf{v}$.
By Lemmas \ref{lem:exist non-poly} and \ref{lem:exist poly}, $f$ must respect $\ell_4.$ As $V$ is dense in $\Rrr^2$, we may choose $\mathbf{w}\in V$ so that $\mathbf{w}$ is to the right of $\ell_4$, left of $\ell_1$, and left of $\ell_2.$ See Figure~\ref{induct}.
But then $f(\mathbf{w})$ must be to the right of $\ell_4'$, to the left of $\ell_1'$, and to the left of $\ell_2'$, which is a contradiction.
\end{proof}

\begin{proof}[Proof of Lemma \ref{lem:dense}]
Let $\mathcal{G}$ be a countable generator set for $P$ which contains at least  three vectors that are pairwise non-parallel. Let $B=\{\vp,\mathbf{q} \},$ and fix $\va\in \mathcal{G}$, where $\va \cdot (\mathbf{p}-\vq)=r\not\in \mathbb{Z}$ .

Without loss of generality, assume that $\vq=\mathbf{0}$. Fix $\va_1,\va_2,\va_3\in \mathcal{G}$ so that $\va_1=\va$, and no two are linearly dependent. Consider the triangular
lattice formed by the lines generated by the $\va_i $ and the point $\mathbf{0}$ in $\mathcal{L}_0(\{\mathbf{0}\}, \{ \va_1,\va_2,\va_3\});$ that
is, all lines with equations $\mathbf{a_i}\cdot\mathbf{x}=z,$ where $i\in \{ 1,2,3\}$
and $z\in \mathbb{Z}$. See the left figure in Figure~\ref{fig6}.

Consider the triangle that contains $\mathbf{p}$. We assume that this triangle is framed by the lines with equations $\va_1\cdot\vx=1$, $\va_2 \cdot\vx=0$ and $\va_3\cdot\vx=0$, as shown in Figure~\ref{fig6}. (The proof can easily be adapted to cover all other possibilities.) By definition, the line $\ell _1$ through $\vp$ with normal $\va_1$ is part of $\mathcal{L}_{0}(B,\mathcal{G}).$ The line $\ell_1$ has as equation $\mathbf{a}_{1}\cdot\mathbf{x}=r$, where $r=\va_1\cdot \vp$, and by assumption $r\in (0,1)$.

The line $\ell _1$ intersects the two sides of the triangle in $\mathbf{p}_{1}$
and $\mathbf{p}_{2}\mathbf{.}$  Following the recursive definition, this implies that $\mathcal{L}_{1}(B,\mathcal{G})$ contains the line
$\ell _{2}$ through $\mathbf{p}_{1}$ with normal $\va_2$. Precisely, $\ell_2$  has equation $%
\mathbf{a}_{2}\cdot\mathbf{x}=\mathbf{a}_{2}\cdot\mathbf{p}_{1}\mathbf{.}$ Similarly, $%
\mathcal{L}_{1}(B,\mathcal{G})$ contains the line $\ell _{3}$  through $\mathbf{p}_{2}$ which has equation $\mathbf{a}_{3}\cdot
\mathbf{x}=\mathbf{a}_{3}\cdot\mathbf{p}_{2}.$ See the right figure in Figure~\ref{fig6}.

\begin{figure}
\begin{tabular}{cc}
\epsfig{figure=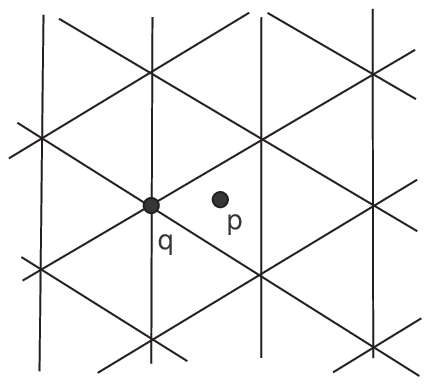, height=2in,width=2in} &
\epsfig{figure=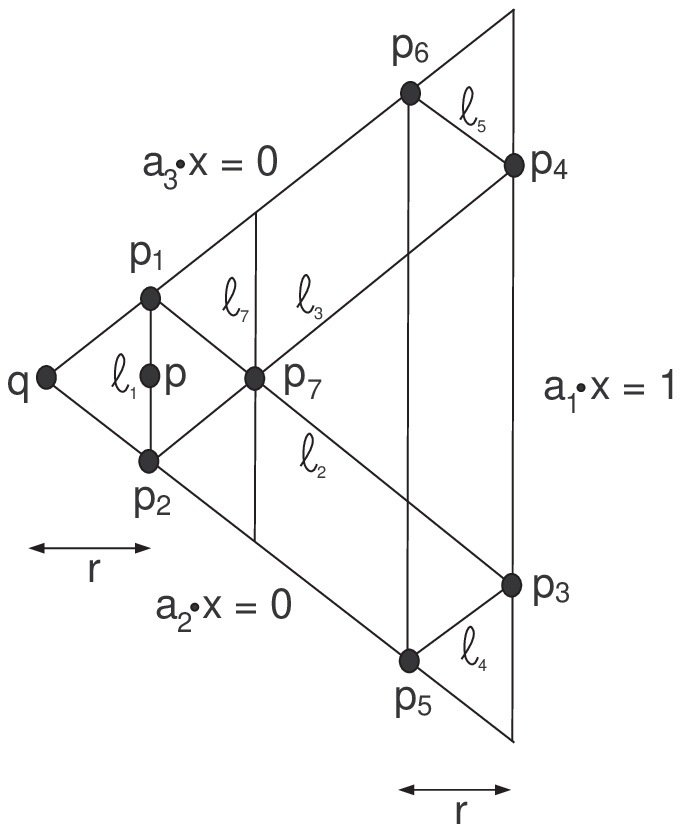, height=2.3in,width=2.3in}
\end{tabular}
\caption{Left: A triangular lattice. Right: Generating the line $\mathbf{a}_3\cdot\mathbf{p}_2=1-r.$}\label{fig6}

\end{figure}

The lines $\ell _{2}$ and $\ell _{3}$ intersect the third side of the
triangle in $\mathbf{p}_{3}$ and $\mathbf{p}_{4},$ generating two lines $%
\ell _{4}$ and $\ell _{5}$ in $\mathcal{L}_{2}(B)$ with
equations $\mathbf{a}_{3}\cdot\mathbf{x}=\mathbf{a}_{3}\cdot\mathbf{p}_{3}$ and $%
\mathbf{a}_{2}\cdot\mathbf{x}=\mathbf{a}_{2}\cdot\mathbf{p}_{4},$ respectively. The
lines $\ell _{4}$ and $\ell _{5}$ intersect with the sides of the triangle
in $\mathbf{p}_{5}$ and $\mathbf{p}_{6},$ generating one line $\ell _{6}$ in
$\mathcal{L}_{3}(B,\mathcal{G})$ with equation $\mathbf{a}_{1}\cdot\vx=\mathbf{a}_{1}\cdot\mathbf{p}_{5}=\mathbf{a}_{1}\cdot\mathbf{p}_{6}.$


By the comparison of similar triangles,
we obtain that $$\mathbf{a}_{1}\cdot\mathbf{p}_{5}=\mathbf{a}_{1}\cdot\mathbf{p}_{6}=%
1-\mathbf{a}_{1}\cdot\mathbf{p}=1-r.$$ Now the parallel lines $\mathbf{a}_{1}\cdot\mathbf{x}=r+z_2,-r+z_2$, $z_2 \in \Zzz$, may be generated from all similar triangles in the lattice in an analogous fashion.

To complete the proof, consider that the lines $\ell_2$ and $\ell_3$ intersect in point $\mathbf{p}_7$, which generates a line $\ell_7$ with normal $\va_1$ as indicated in the right figure in Figure~\ref{fig6}. Since the triangle formed by $\mathbf{p}_1,\mathbf{p}_2$, and $\mathbf{p}_7$ is half of a parallelogram formed by  $\mathbf{0},\mathbf{p}_1,$ $\vp_2$ and $\mathbf{p}_7$, it follows that $\ell_7$ has equation $\va_1\cdot\vx=2r$.
This process can be repeated to obtain all the lines $\va_1\cdot\vx =z_1r+z_2$, $z_1,z_2\in\Zzz$.

If $r$ is irrational, then the set $%
\{z_1r+z_2:z_1,z_2\in \mathbb{Z}\}$ is dense in $\mathbb{R}$ (this is a result from folklore which can be proved by using the pigeonhole principle). That completes the proof of the lemma.
\end{proof}

\section{Respecting lines when $P$ is a polygon; proof of Lemma \ref{lem:exist poly}}\label{sec:proofs2}

In this section, we consider norm-derived metrics on $\Rrr^2$ whose shape is a polygon. Specifically, fix $d\in \Omega$ so that $d$ is a polygonal metric which is not a box metric, and let $P$ be the shape of $d$.
Let $V$ be a countable dense set in $\Rrr^2$ and let $f:V\rightarrow V$ be a step-isometry. Since $P$ is a polygon, its generator set is finite. Let $\mathcal{G}$ be the generator set of $P$ where each direction is represented by only one vector, so we have that $P=\{\vx: \forall\,{\va\in\mathcal{G}},\,-1\leq \va\cdot \vx \leq 1\}$, and for all $\vx,\vy\in\Rrr$,
\[
d(\vx,\vy)=\max\{ |\va\cdot (\vx-\vy)| : \va\in\mathcal{G}\}.
\]
Recall that we use $F(P,\va)$ to denote the face of $P$ with
normal $\mathbf{a}\in\mathcal{G},$ where $\mathbf{a}$ is pointing away from the
centre of $P$.
We will show that any line through a point $\vv\in V$ and parallel to one of the sides of $P$ must be respected, and its image is a line through $\vv'=f(\vv)$ parallel to one of the sides of $P$.

Before we prove Lemma~\ref{lem:exist poly}, we need the following technical lemma. The lemma can be understood as follows. Given a point $\vp\in\Rrr$, and a line $\ell $ through $\vp$ with normal vector  $\va\in\mathcal{G}$, it is straightforward to define two polygons which are both similar to $P$, so that $\ell$ is the only line through $\vp$ with normal vector in $\mathcal{G}$ which separates the two polygons. Namely, consider two copies of $P$ side-by-side; that is, they intersect in a face defined by  $\va$, place them such that $\vp$ lies on the shared face, and then move them apart so that they do not intersect, but no other line through $\vp$ with normal vector in $\mathcal{G}$ fits between the two polygons. (See Figure \ref{fig3}.) However, to define these separating polygons one must be able to precisely define the distances.

The following lemma shows that separating polygons can also be defined if only rounded distances are given. This is achieved by making the polygons very large. Before stating the lemma, we introduce the following notation for the rounded distance. For all $\vx,\vy\in\Rrr$, define
\[
D(\mathbf{x},\mathbf{y})=\left\lfloor d(\mathbf{x}
,\mathbf{y})\right\rfloor .
\]

\begin{lemma}
\label{lem:techical}
There exist positive integers $M$ and $m$ such that, given any points $$\mathbf{p},\mathbf{z},\mathbf{z}^*,\vu, \vu^*,\vy,\vy^* \in \mathbb{R}^{2}$$ which satisfy the conditions (\ref{cond:distances}) below,
 there exists a \em{unique} line through $\mathbf{p}$ parallel to one of the sides of $P$
which separates the sets $\{\mathbf{u}:d(\mathbf{u},\mathbf{z})\leq M\}$ and $\{\mathbf{u}:d(\mathbf{u},\mathbf{z}^{* })\leq M\}.$
\begin{eqnarray}
\label{cond:distances}
D(\mathbf{u},\mathbf{u}^{* }) &=&D(\mathbf{y},\mathbf{y}^{* })=0,
\notag\\
D(\mathbf{u},\mathbf{p}) &=&D(\mathbf{u}^{* },\mathbf{p})=D(\mathbf{y},
\mathbf{p})=D(\mathbf{y}^{* },\mathbf{p})=m, \notag\\
D(\mathbf{u},\mathbf{y}) &= &D(\mathbf{u}^{* },\mathbf{y}^{*})=2m, \notag\\
D(\mathbf{z},\mathbf{u}) &=&D(\mathbf{z},\mathbf{y})=D(\mathbf{z}^{* },
\mathbf{u}^{* })= D(\mathbf{z}^{* },\mathbf{y}^{* })=M-1,\notag\\
D(\mathbf{z},\mathbf{p}) &=& D(\mathbf{z}^{* },\mathbf{p})=M,\notag\\
D(\mathbf{z},\mathbf{z}^{* })&=&2M.\notag\\
\mbox{None of the distances between }\vu,\vu^*,&\vy,\vy^*,&\vz,\vz^*\mbox{ and }\vp\mbox{ are integer.}
\end{eqnarray}
\end{lemma}

\begin{proof} Let $m$ and $M$ be positive integers that satisfy the following conditions. Note that  the conditions impose only lower bounds, and thus, $m$ and $M$ can always be chosen so that the conditions hold by choosing them  sufficiently large.
\begin{enumerate}
\item[({\sl a})] Let $\theta_{\min}$ be the minimum angle between any two vectors in $\mathcal{G}$, and let $c>0$ be so that for all $\vx\in \Rrr$, $d(\vx,\vy)\geq c||\vx,\vy||_2$. (See Section \ref{sec:norm} to see that such a constant exist.)
We must choose $m$ so that $cm>1/\sin(\theta_{\min})$.

\item[({\sl b})]  Let $m$ be chosen such that, for all $\va\in\mathcal{G}$,  the shortest path along the boundary of $P$
from a point on the line $\ell ^{+}$ with equation $\mathbf{a}\cdot \mathbf{x}
=1/m,$ and a point on the line $\ell ^{-}$ with equation $\mathbf{a}\cdot \mathbf{x}=-1/m$
has length at most $1$ (where length is measured according to metric $d$). See Figure~\ref{fig2}.

\begin{figure} [h!]
\begin{center}
\epsfig{figure=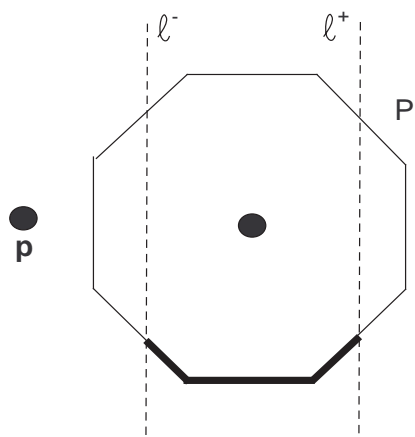,width=3in,height=2.5in}
\caption{Condition (b): the length of the bold path along the boundary is at most $1/m$.} \label{fig2}
\end{center}
\end{figure}

\item[({\sl c})]  The integer $M$ is such that, for any $\va\in \mathcal{G}$, the region strictly between the lines $\ell ^{-}$ with equation $\mathbf{a}\cdot \mathbf{x}
=1-1/M,$ and the line $\ell ^{0}$ with equation $\mathbf{a}\cdot \mathbf{x}
=1$ (an extension of the face $F(P,\va)$), contains no vertex of $P$. See Figure~\ref{jfig_new3}.

\item[({\sl d})] $2<m<M<2m$, and $M$ is  such that, for any $\va\in \mathcal{G}$, the following holds. Let lines $\ell^{-}$ and $\ell^{0}$ as defined under (c), and let $P'$ be a similar copy of $P$ enlarged by a factor $1+1/M$.
Let $\vq$ be a vertex of $P$ on $\ell^{0}$, and let $\vv$ be the point on the intersection of the boundary of $P$ and  $\ell^-$ which is closest to $\vq$. See Figure~\ref{jfig_new3}. Then for any point $\vp$ in the region enclosed by $P$, $P'$, $\ell^{-}$ and $\ell^{0}$, $d(\vv,\vp)\leq \frac{1}{2}$ and $d(\vv,\vq)\leq \frac12$.
Note that this implies that $d(\vp,\vv)\leq \frac{m}{M}$, and the same holds for $d(\vp,vq)$.

\begin{figure} [h!]
\begin{center}
\epsfig{figure=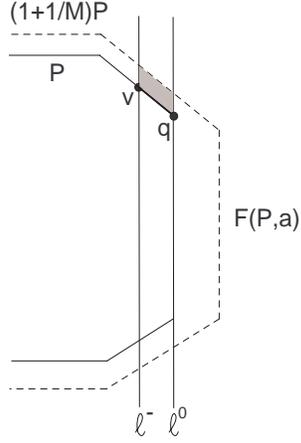,width=2in, height=2.5in}
\caption{Condition (c): no vertices of $P$ lie between $\ell^-$ and $\ell^0$. Condition (d): limits the distance from any point in the shaded area to the points $\vv$ and $\vq$.}\label{jfig_new3}
\end{center}
\end{figure}

\item[({\sl e})] The integer $M$ is large enough so that points can be chosen so that conditions (\ref{cond:distances}) hold.

\end{enumerate}

Note that conditions (a) and (b) only establish lower bounds on $m$, while (e) and the second part of (d) establish a lower bound on $M$ which is independent of $m$. Only the first part of (d) that requires $m<M<2m$ involves both $m$ and $M$. This condition can be satisfied together with the others by taking $m$ and $M$ large enough.

Let $\mathbf{p},\mathbf{z},\mathbf{z}^*,\vu,\vu^*,\vy,\vy^*\in \mathbb{R}^{2}$ be such that the conditions in (\ref{cond:distances}) hold.
Let
\[
P_{\vz}=B_M(\vz)\mbox{ and }
P_{\vz^{*}}=B_M(\mathbf{z}^{*}).
\]
See Figure~\ref{fig3}. Hence, $P_{\vz}$ and $P_{\vz^*}$ are the polygons, similar to $P$, which form the balls of radius $M$ around $\mathbf{z}$ and $\mathbf{z^*}$, respectively. We will show first that there exists a line through $\mathbf{p}$ parallel to one of the sides of $P$ which  separates $P_{\vz}$ and $P_{\vz^*}$.

\begin{figure} [h]
\begin{center}
\epsfig{figure=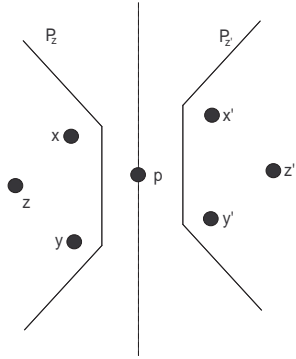,width=2in,height=2in}
\caption{The balls $P_{\mathbf{z}}$ and $P_{\mathbf{z}^*}$.}\label{fig3}
\end{center}
\end{figure}

Let $\mathbf{a}\in \mathcal{G}$ be the vector that determines $d(\mathbf{z},\mathbf{z}^{* })$ so that
\[
d(\mathbf{z},\mathbf{z}^{* })=\mathbf{a}\cdot (\mathbf{z}
^{* }-\mathbf{z}),
\]
and for all $\va_{\mathcal{G}}\in \mathcal{G}$,
\[
|\mathbf{a}_{\mathcal{G}}\cdot (\mathbf{z}^{* }-\mathbf{z})|\leq d(\mathbf{z},
\mathbf{z}^{* }).
\]

This implies that the line
segment connecting $\mathbf{z}$ and $\mathbf{z}^{* }$ intersects $P_{
\mathbf{z}}$ in $F(P_{\mathbf{z}},\mathbf{a})$ and $P_{\vz^{* }}$ in $
F(P_{\mathbf{z}^{* }},-\mathbf{a})$. See also Section \ref{sec:norm} and Figure \ref{fig4}.

Let $\ell _{1}$
and $\ell _{2}$ be the lines extending $F(P_{\vz},\mathbf{a})$ and $
F(P_{\vz^{* }},-\mathbf{a}),$ respectively. Precisely, $\ell_1$ is the line with equation $\mathbf{a}\cdot \vx =\mathbf{a}\cdot\mathbf{z}+M$, and $\ell_2$ has equation $\mathbf{a}\cdot\vx =\mathbf{a}\cdot\mathbf{z^*}-M$. Since $D(
\mathbf{z},\mathbf{z}^{* })=\lfloor d(
\mathbf{z},\mathbf{z}^{* })\rfloor=2M$ we have that $P_{\mathbf{z}}$ and $P_{
\mathbf{z}^{* }}$ do not intersect, and thus, $\ell_1$ lies to the left of $\ell_2$.  If $\mathbf{p}$ lies
between $\ell _{1}$ and $\ell _{2},$ then the line with equation $\mathbf{a}\cdot\vx =\mathbf{a}\cdot \mathbf{p}$ separates $P_{\vz}$ and $P_{\vz^*}$, as in Figure \ref{fig3}, and we are done.

Suppose now, by contradiction, that this is not the case. See Figure~\ref{fig5}.
Without loss of generality, suppose that $\vp$ lies to the left ($\mathbf{z}$-side) of $\ell _{1}.$ By definition $\mathbf{u}^*$ and $\mathbf{y}^{* }$ lie inside $P_{
\mathbf{z}^{* }},$ and thus, to the right of $\ell _{2}$. Since $d(\mathbf{u},\mathbf{u}
^{* })<1$ and $d(\mathbf{y},\mathbf{y}^{* })<1,$ we must have that
$\mathbf{u}$ and $\mathbf{y}$ both lie to the right ($\mathbf{z^*}$-side) of the line $\ell _{3}$, which is an integer parallel of $\ell_2$
defined by $\mathbf{a}\cdot \mathbf{x}=\mathbf{a}\cdot \mathbf{z^*}-M-1.$ Note that $\ell_3$ is such that $\ell_2$ and $\ell_3$ are distance 1 apart in the polygon metric $d$. Since $d(\vp,\mathbf{z}
^{* })<M+1,$ $\mathbf{p}$ must lie to the right of $\ell _{3}.$

\begin{figure} [h!]
\begin{center}
\epsfig{figure=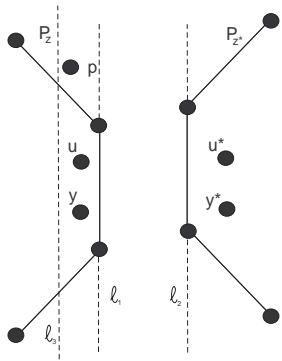,width=2in,height=2in}
\caption{The point $\mathbf{p}$ must lie to the right of $\ell _3$.}\label{fig5}
\end{center}
\end{figure}

Consider $P_{\mathbf{p}}=B_m(\vp)$, the $m$-ball around $\vp$.
Let $\mathbf{v}$ be the point on the intersection of the boundary of $P_{\vz}$ and $\ell_3$   and let $\mathbf{q}$ be the vertex of $P_{\vz}$ which is an endpoint of $F(P_z,\mathbf{a})$ (see Figure \ref{fig2}). (Following the definition, there are actually two choices for both $\vq$ and $\vv$; in each case, we choose the point closest to $\vp$.) By condition ({\sl c}) on $M$, no vertex of $P_{\vz}$ lies between $\ell_3$ and $\ell_1$, and thus,
 the line segment from $\vq$ to $\vv$ is part of the boundary of $P_{\vz}$.
By condition ({\sl d}), this line segment has length, measured according to $d$, at most $m-1$.

Since $d(\vp,\vz)\leq M+1$, $\vp$ lies within a strip of width 1 along the boundary of $P_{\vz}$. By  condition ({\sl d}), $d(\vv,\vp)\leq m$ and $d(\vv,\vq)\leq m$, and so $\mathbf{v}$ and $\vq$ lie inside $P_{\mathbf{p}}$, and, by convexity, so does the entire piece of the boundary of $P_{\vz}$ between $\vv$ and $\vq$.

\begin{figure} [h!]
\begin{center}
\epsfig{figure=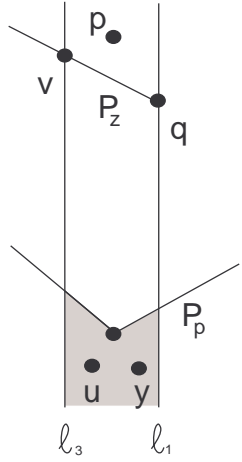}
\caption{The position of $\mathbf{u}$ and $\mathbf{y}$. The shaded area is the only area between lines $\ell_1$ and $\ell_3$ outside $P_{\vp}$ and inside $P_{\vz}$.} \label{fig8}
\end{center}
\end{figure}

We therefore have that there is only one connected region between $\ell _{1}$ and $
\ell _{3}$ outside $P_{\mathbf{p}}$ and inside $P_{\mathbf{z}}.$ See shaded area in Figure~\ref{fig8}.
Hence, $\mathbf{u}$ and $\mathbf{y}$ must both lie in the unique region
outside $P_{\mathbf{p}}$ and between $\ell _{1}$ and $\ell _{3}.$
In particular, they both lie on the same side of $P_{\mathbf{p}}$.
However, since $d(\mathbf{u},\mathbf{p})< m+1$ and $d(\mathbf{y},
\mathbf{p})<m+1,$ $\mathbf{u}$ and $\mathbf{y}$ are both within distance $1$
(in polygon distance) of $P_{\mathbf{p}}.$ By condition (b) on the choice of $m$, the length of
the piece of the boundary of $P_{\mathbf{p}}$ between lines $\ell _{1}$ and $
\ell _{3}$, measured in polygon distance, is at most $m$. Therefore,
\[
d_{P}(\mathbf{u},\mathbf{y})\leq m+2<2m,
\]
which contradicts the assumption about the choice of $\mathbf{u}$ and $\mathbf{y}$.
Hence, $\mathbf{p}$ lies between $\ell_1$ and $\ell_2$ as in Figure \ref{fig3}, and thus, the line $\ell$ with equation $\mathbf{a}\cdot \mathbf{x} =\mathbf{a}\cdot \mathbf{p}$ separates $P_{\mathbf{z}}$ and  $P_{\mathbf{z^*}}$.

Next, we prove that the line $\ell$ is unique with the given property. Suppose by contradiction that there are two distinct lines $\ell^1$ and $\ell^2$, with normal vectors $\va_1$ and $\va_2$ in $\mathcal{G}_P$, respectively, so that both $\ell^1$ and $\ell^2$ contain $\vp$ and separate $P_{\mathbf{z}}$ and  $P_{\mathbf{z}^*}$. In particular, both lines must separate $\vu$ and $\vu^*$. Now, by the conditions in (\ref{cond:distances}), $d(\vu,\vu^*)<1$, $d(\vp,\vu)\geq m$ and $d(\vp, \vu^*)\geq m$. Therefore, the Euclidean length of the line segments $\vp\vu$ and $\vp\vu^*$ is at least $cm$, where $c$ is the constant defined as in condition (a). This implies that the angle $\phi$ between the line segments $\vp\vu$ and $\vp\vu^*$ is such that $\sin(\phi)<1/cm$. The angle between $\ell^1$ and $\ell^2$ must be smaller than the angle between $\va_1$ and $\va_2$, but by condition (a) the angle between any two vectors in $\mathcal{G}$ has sine value at least $1/cm$. This gives a contradiction. \end{proof}

\begin{proof}[Proof of Lemma \ref{lem:exist poly}]
Assume without loss of generality that the origin $\mathbf{0}\in V$, and that $f(\mathbf{0})=\mathbf{0}$. Fix $\va\in \mathcal{G}$. Let $M$ and $m$ be as in Lemma \ref{lem:techical}, let $\vp=\mathbf{0}$, and choose $\vz,\vz^*,\vu,\vu^*,\vy,\vy^*$ such that $ \|\vz\|_P=\va\cdot\vz$ and $\|\vz^*\|_P=-\va\cdot \vz^*$, and the conditions (\ref{cond:distances}) of Lemma \ref{lem:techical} hold. Hence, the sets $P_{\vz}=B_M(\vz)$ and $P_{\vz^*}=B_M(\vz^*)$ are separated by the line $\ell$ with equation $\vx\cdot \va =0$.

Since $f$ is a step-isometry, and conditions (\ref{cond:distances}) only refer to rounded distances, the conditions (\ref{cond:distances}) also hold for the images $f(\vp)=\mathbf{0},f(\vz),f(\vz^*),f(\vu),f(\vu^*),f(\vy),$ and $f(\vy^*)$.
Therefore, by Lemma \ref{lem:techical} there exists a vector $\va'\in \mathcal{G}$ so that the line $\ell '$ with equation $\va'\cdot \vx =0$ separates the $M$-balls around $f(\vz)$ and $f(\vz^*)$.
Precisely, $\va'\cdot f(\vz)< 0$ and $\va'\cdot f(\vz^*)>0$.

To complete the proof, we claim that the map $f$ respects the line $\ell$ with equation $\vx\cdot \va =0$, and the image of $\ell$ under the line map is the line $\ell'$.

To see this, fix any $\vw\in V$ to the left of $\ell$, so $\va\cdot\vw<0$. Choose integer $\overline{M}\ge M+1$, and points $\bar{\vz}^*,\bar{\vu},\bar{\vu}^*,\bar{\vy},\bar{\vy}^*$ so that $D(\vw,\bar{\vz})\leq \bar{M}$ and $\bar{\vz},\bar{\vz}^*,\bar{\vu},\bar{\vu}^*,\bar{\vy},\bar{\vy}^*,\vp$ satisfy conditions (\ref{cond:distances}) with $M$ replaced by $\overline{M}$. Moreover, the new points are chosen so that $P_{\vz}$ is contained in $P_{\bar{\vz}}=B_{\overline{M}}(\bar{\vz})$ and $P_{\vz^*}$ is contained in $P_{\bar{\vz^*}}=B_{\overline{M}} (\bar{\vz}^*)$, and so that  the sets $P_{\bar{\vz}}$ and $P_{\bar{\vz}^*}$ are separated by the line $\ell$.

By definition, $\vw\in P_{\bar{\vz}}$. By Lemma~\ref{lem:techical} there must be a line $\hat{\ell}$ through $f(\vp)=\mathbf{0}$ with normal vector in $\mathcal{G}$ which separates the $\overline{M}$-balls around $f(\bar{\vz})$ and $f(\bar{\vz}^*)$. The line $\hat{\ell}$ also separates the $M$-balls around $f(\vz)$ and $f(\vz^*)$, and thus, again by Lemma \ref{lem:techical}, line $\hat{\ell}$ must equal line $\ell'$. Therefore, $f(\vw)$ must lie on the left (that is, the $f(\vz)$-side) of $\ell'$. A similar argument holds for vertices to the right of $\ell$. This completes the proof.
\end{proof}

\section*{Acknowledgements} We would like to thank the anonymous referee for useful comments which improved the correctness and presentation of the paper.

\end{document}